\documentclass[12pt]{article}
\usepackage{subfigure}
\usepackage{tikz}

\usetikzlibrary{graphs,graphs.standard,calc}
\usepackage[utf8]{inputenc}
\usepackage{svg}
\usepackage{comment}
\usepackage{enumerate}

\usepackage[utf8]{inputenc}
\usepackage[english]{babel}
\usepackage{amsmath}
\usepackage{times}
\usepackage{graphicx}
\usepackage{float}
\usepackage[margin=1in]{geometry}
\usepackage{blindtext}
\usepackage{amssymb}
\usepackage{tikz}
\usepackage{amsthm}
\usepackage{mathtools}
\usepackage{textcmds}
\usepackage{url}

\usepackage{comment}
\usepackage[mathlines]{lineno}
\newtheorem{theorem}{Theorem}[section]
\newtheorem{corollary}[theorem]{Corollary}

\newtheorem{lemma}[theorem]{Lemma}
\newtheorem{conjecture}[theorem]{Conjecture}
\newtheorem*{problem*}{Problem}
\theoremstyle{remark}

\theoremstyle{definition}
\newtheorem{definition}[theorem]{Definition}

\newcommand{\RR}{\mathbb R}

\newcommand{\ZZ}{\mathbb Z}
\newcommand{\cD}{\mathcal D}

\newcommand{\DL}{\mathcal D^{\mathcal L}}

\DeclareMathOperator{\vol}{vol}
\newcommand{\abc}[1]{}

\DeclareSymbolFont{bbold}{U}{bbold}{m}{n}
\DeclareSymbolFontAlphabet{\mathbbold}{bbold}

\title{Expansion in Distance Matrices}
\author{John Byrne\thanks{Department of Mathematical Sciences, University of Delaware, \texttt{jpbyrne@udel.edu}} \and Jacob Johnston\thanks{Department of Mathematics \& Statistics, Villanova University, \texttt{jjohns80@villanova.edu}} \and Carl Schildkraut\thanks{Department of Mathematics, Stanford University, \texttt{carlsch@stanford.edu}} \and Michael Tait\thanks{Department of Mathematics \& Statistics, Villanova University,  \texttt{michael.tait@villanova.edu}.}}
\date{\today}

\begin{document}

\maketitle

\begin{abstract}
    The normalized distance Laplacian matrix $\mathcal{D}^{\mathcal{L}}(G)$ of a graph $G$ is a natural generalization of the normalized Laplacian matrix, arising from the matrix of pairwise distances between vertices rather than the adjacency matrix. Following the motif that this matrix behaves quite differently to the normalized Laplacian matrix, we show that both the spectral gap and Cheeger constant of $\mathcal{D}^{\mathcal{L}}(G)$ are bounded away from $0$ independently of the graph $G$. The spectral result holds more generally for finite metric spaces.
\end{abstract}

\section{Introduction}
Given a connected graph $G$, one may define the {\em distance matrix} $\mathcal{D}$ for $G$ where the $ij$'th entry of $\mathcal{D}$ is the distance in $G$ from vertex $i$ to vertex $j$. This matrix was popularized in seminal work by Graham and Lov\'asz \cite{GL} that was motivated by a problem out of Bell Labs regarding telephone networks \cite{GP}. The origins of the distance matrix go back to sporadic work long before this popularization, for example \cite{cayley, YH}. Since then, the matrix has been widely studied; a nice survey is given by Aouchiche and Hansen \cite{AH}. 

One may think of $\mathcal{D}$ as a weighted complete graph where the weights are determined by distances in the metric space induced by $G$. In this framework, problems on distance matrices of graphs arise as special cases of problems on distance matrices in arbitrary finite metric spaces. For some background on this approach, see the text \cite{DL}. One may also consider other matrices that capture information about graph distance besides $\mathcal{D}$. Motivated by various applications, matrices such as distance Laplacian and the signless distance Laplacian have been introduced \cite{AH2}. 

In this paper, we will study the {\em normalized distance Laplacian} matrix of a graph $G$ (or, more generally, of a finite metric space). This matrix was introduced by Reinhart \cite{reinhart} and can be thought of as the classical normalized Laplacian on a weighted complete graph with weights given by distances in $G$ (formal definitions to follow). From this perspective, the study of the normalized distance Laplacian follows a long line of work on normalized Laplacian matrices on graphs and weighted graphs, see for example \cite{fan}.

The study of normalized Laplacian matrices is a fundamental problem in spectral graph theory (the text \cite{fan} has been cited over ten thousand times). Thus, when specializing to this subproblem (where the weights must come from distances in a graph), an overarching question is how the subfamily behaves differently than the general setting. Recent work on distance matrices indicates that they behave quite differently from more classical matrices like the adjacency matrix or the normalized Laplacian. For example, in \cite{steinerberger}, it is shown that the Perron vector of $\mathcal{D}$ must be relatively ``close" to a constant vector, which is of course very different from the Perron vector of an adjacency matrix for many graphs \cite{CD}. It is a classical result that the largest eigenvalue of the normalized Laplacian matrix is equal to $2$ if and only if the graph has a bipartite component. In \cite{reinhart}, it was shown that for graphs on more than $2$ vertices, the largest eigenvalue of the normalized distance Laplacian is always strictly less than $2$ (see also \cite{JT}).

Our theorems follow the same theme: we show that the normalized distance Laplacian matrix behaves in a nonstandard way with respect to ``expansion." We now state our theorems formally, after giving relevant definitions.

Given a connected graph $G$ and $u\in V(G)$, let $t(u)=\sum_{v\in V(G)}d(u,v)$ be the {\em transmission} of the vertex $u$, and let the \textit{transmission matrix} $T(G)$ be $\mathrm{diag}(t(u):u\in V(G))$. Let $\mathcal D(G)$ be the distance matrix as before with $\mathcal{D}_{u,v}=d(u,v)$. The \textit{normalized distance Laplacian} of $G$ is
\[\mathcal D^\mathcal L(G)=T(G)^{-1/2}(T(G)-\mathcal D(G))T(G)^{-1/2}.\]
We will denote the $n$ (real) eigenvalues of $\mathcal D^\mathcal L(G)$ by $0=\partial_1<\ldots<\partial_n \leq 2$, and we call $\partial_2$ the \textit{spectral gap} of $\mathcal D^\mathcal L$. 

Given $S\subseteq V(G)$, define the \textit{volume} of $S$ as $\vol S=\sum_{u\in S}t(u)$ and define
$$h(S)=\frac{\sum_{u\in S,v\in\overline{S}}d(u,v)}{\min\{\vol S,\vol\overline S\}}.$$
The \textit{Cheeger constant} of $\mathcal D^\mathcal L$ is $h_G=\min_{\emptyset\ne S\subseteq V} h(S)$. Note that these definitions agree with the classical definitions of the second eigenvalue of the normalized Laplacian and with the Cheeger constant of the weighted complete graph with weights given by distances in $G$, where by classical Cheeger constant we mean with respect to edge-expansion as in \cite{fan} (this constant is sometimes called the {\em conductance}). 

Our first theorem gives a lower bound on the Cheeger constant of $\mathcal{D}^\mathcal{L}$ and characterizes the cases of equality.

\begin{theorem}\label{cheeger constant theorem}
    Let $G$ be a connected graph on $n$ vertices and let $h_G$ be the Cheeger constant for $\mathcal{D}^\mathcal{L}$. 
    \begin{enumerate}[(i)]
        \item If $n$ is even, then $h_G \geq \frac{n}{3n-4}$ with equality if and only if $G$ is isomorphic to $K_{n/2, n/2}$.
        
        \item If $n$ is odd, then $h_G \geq \frac{n+1}{3n-5}$. For $n\neq 5$, equality holds if and only if $G$ is isomorphic to a graph obtained from $K_{\lfloor n/2\rfloor, \lceil n/2 \rceil}$ by embedding at most $n-1$ edges in the larger part.
    \end{enumerate}
\end{theorem}

We have excluded the $n=5$ case from the equality characterization in part (ii) above because there are three additional equality cases, up to isomorphism. They are the path on five vertices, a four-cycle with an additional pendant edge, and a four-cycle $uvwx$ with a pendant edge $xy$ and an additional edge $uw$.

Theorem \ref{cheeger constant theorem} says that the Cheeger constant for $\mathcal{D}^\mathcal{L}$ is always larger than $1/3$. In contrast, the classical Cheeger constant on a connected $n$-vertex graph can go to $0$ with $n$. For example, two cliques of order $n/3$ joined by a path of length $n/3$ has a Cheeger constant that decays quadratically in $n$. 
There is an intimate relationship between the Cheeger constant and the spectral gap \cite{fan}. Cheeger's inequality gives 
\begin{equation}\label{cheeger inequality}
   \frac{h_G^2}{2} \leq \partial_2 \leq   2h_G.
\end{equation}

Since classical Cheeger constants of connected graphs can go to $0$, Cheeger's inequality implies that the spectral gap of the normalized Laplacian can also go to $0$. With the same example as before, two cliques joined by a path of length $n/3$, the spectral gap of the normalized Laplacian is asymptotic to $\frac{54}{n^3}$ (and in fact this is asymptotically the minimum possible \cite{ACTT, aldous-fill-2014}). On the other hand, Theorem \ref{cheeger constant theorem} with Cheeger's inequality implies that the spectral gap of a normalized distance Laplacian matrix is always at least $\frac{1}{18}$. Our second theorem is the following improvement on this bound.

\begin{theorem}\label{spectral gap theorem}
    Let $G$ be a connected graph and $\partial_2$ the spectral gap of its normalized distance Laplacian matrix. Then 
    \[
    \partial_2 \geq \frac{9 - 4\sqrt{2}}{7} \approx 0.478.
    \]
\end{theorem}

In fact, we believe that this lower bound can be improved to $2/3$.

\begin{conjecture}\label{2/3 conjecture}
    Let $G$ be a connected graph and $\partial_2$ the spectral gap of its normalized distance Laplacian matrix. Then 
    \[
    \partial_2 \geq 2/3.
    \]
\end{conjecture}

As evidence towards Conjecture \ref{2/3 conjecture}, we provide the following, which demonstrates a special case as well as an improvement in some settings.

\begin{theorem}\label{2/3 theorem} Conjecture \ref{2/3 conjecture} holds when $G$ is a Cayley graph on an abelian group. Moreover, if said abelian group has odd order, then we have the stronger bound $\partial_2>0.718$.    
\end{theorem}

The proofs of Theorems \ref{spectral gap theorem} and \ref{2/3 theorem} apply to the more general setting of finite metric spaces; see Theorems \ref{spectral gap theorem metric spaces} and \ref{2/3 theorem metric spaces}. The remainder of the paper is structured as follows. In Section \ref{preliminaries section} we give notation and preliminaries. In Section \ref{cheeger constant section} we prove Theorem \ref{cheeger constant theorem}. In Section \ref{eigenvalue section} we prove Theorems \ref{spectral gap theorem} and \ref{2/3 theorem}.

\section*{Acknowledgments}

The third author was supported by NSF Graduate Research Fellowship Program DGE-2146755. The fourth author was partially supported by National Science Foundation grant DMS-2245556. This material is partially based on work supported by the National Science Foundation grant DMS-1928930 while the third author was in residence at the Simons Laufer Mathematical Sciences Institute. We thank Nitya Mani for helpful advice on semidefinite programming which enabled us to prove Theorem \ref{2/3 theorem}.

\section{Notation and preliminaries}\label{preliminaries section}

When considering a partition $S,\overline S$ of $V(G)$ as in the definition of $h_G$ we will always write $s$ for $|S|$ and $\overline s$ for $|\overline S|$. The \textit{Cheeger partition} of $G$ is the partition $S,\overline S$ such that $h_G=h(S)$. Let $e(S)$ be the number of edges with both endpoints in $S$, let $G[S]$ be the subgraph of $G$ induced by $S$, and for $R\subseteq V(G)-S$, let $G[S,R]$ be the bipartite subgraph of $G$ induced by $S$ and $R$. Given $T\subseteq V(G)$, we write $D(S,T)$ for $\sum_{u\in S,v\in T}d(u,v)$.

Observe that
$$\mathcal D^\mathcal L(u,v)=\begin{cases}
    1,&u=v\\
    -\frac{d(u,v)}{\sqrt{t(u)t(v)}},&u\ne v.
\end{cases}$$

Let $G$ be a graph with normalized distance Laplacian $\mathcal D^\mathcal L$, and let $0=\partial_1<\partial_2<\ldots<\partial_n$ be the eigenvalues of $\mathcal D^\mathcal L$. We will make use of the Courant--Fischer variational characterization of $\partial_2$, which we write in the following convenient form.

\begin{lemma}\label{Courant-Fischer} One has
\[\partial_{2}=\min\limits_{y\perp T\mathbf{1}}\frac{\sum_{u,v}d(u,v)(y_u-y_v)^2}{2\sum_{u}t(u)y_u^2 }=\max\limits_{z\in\RR^{V(G)}}\min\limits_{y\perp z}\frac{\sum_{u,v}d(u,v)(y_u-y_v)^2}{2\sum_{u}t(u)y_u^2},\]
where the minima and maximum range only over nonzero vectors and the sum is over ordered pairs $u,v$.
\end{lemma}
\begin{proof}
    The distance Laplacian $T-\cD$ is positive semidefinite with $(T-\cD)\mathbf 1=0$. So, $\DL$ is positive semidefinite as well with $\DL(T^{1/2}\mathbf 1)=0$. Since $\DL$ is symmetric, the Courant--Fischer theorem gives us
    \[\partial_2=\max_w\min_{x\perp w}\frac{x^\intercal\DL x}{x^\intercal x}=\min_{x\perp T^{1/2}\mathbf 1}\frac{x^\intercal\DL x}{x^\intercal x}.\]
    We now substitute $y=T^{-1/2}x$. Noting that $x\perp T^{1/2}\mathbf 1$ if and only if $T^{-1/2}x\perp T\mathbf 1$, we obtain
    \[\partial_2=\max_z\min_{y\perp z}\frac{y^\intercal (T-\cD)y}{y^\intercal Ty}=\min_{y\perp T\mathbf 1}\frac{y^\intercal(T-\cD)y}{y^\intercal Ty}.\]
    The denominator of this expression is $\sum_ut(u)y_u^2$, so it remains to show that the numerator is $\frac12\sum_{u,v}d(u,v)(y_u-y_v)^2$. Indeed, we have
    \begin{align*}
        \sum_{u,v}d(u,v)(y_u-y_v)^2&=2\sum_{u,v}d(u,v)y_u^2-2\sum_{u,v}d(u,v)y_uy_v\\
        &=2\sum_{u}y_u^2t(u)-2y^\intercal\cD y\\
        &=2y^\intercal Ty-2y^\intercal \cD y=2y^\intercal(T-\cD)y.\qedhere
    \end{align*}
\end{proof}

Below we state Cheeger's inequality for the normalized distance Laplacian, which may be derived from Equation 0.2 in \cite{friedland1992lower}.

\begin{theorem} \label{Cheeger inequality}
    For any graph $G$ we have $\frac{h_G^2}{2}\le\partial_2(G).$
\end{theorem}

We also observe the simple identity
\begin{equation}\label{volume equality}
\vol S=\sum_{u\in S}t(u)=\sum_{u,v\in S}d(u,v)+\sum_{u\in S,v\in\overline{S}}d(u,v)
\end{equation}
which will be used repeatedly in Section \ref{cheeger constant section}.

Finally, we give a formal definition of Cayley graphs and one result about their spectra, which we use in Section \ref{eigenvalue section} to prove Theorem \ref{2/3 theorem}.

\begin{definition} Let $\Gamma$ be a group and $S\subset\Gamma$ be a subset of $\Gamma$ closed under taking inverses. The \emph{Cayley graph} $\operatorname{Cay}(\Gamma,S)$ is the graph with vertex set $\Gamma$ and edge set $\{(g,gs):g\in\Gamma,s\in S\}$. (The fact that $S$ is closed under inverses guarantees that $\operatorname{Cay}(\Gamma,S)$ is an undirected graph, i.e.~that $(x,y)$ is an edge of our graph if and only if $(y,x)$ is.)
\end{definition}

\begin{lemma}\label{Cayley spectrum} Let $\Gamma$ be an abelian group, written additively, let $f\colon\Gamma\to\mathbb C$ be any function, and consider a matrix $M_f$ defined by $(M_f)_{uv}=f(u-v)$. An eigenbasis of $M_f$ can be described as follows: for each character $\chi$ of $\Gamma$ (i.e.~for each homomorphism $\chi\colon\Gamma\to\mathbb C^\times$), the vector $(\chi(v))_v$ is an eigenvector of $M_f$ with eigenvalue
\[\sum_{v\in\Gamma}f(v)\chi(v).\]
\end{lemma}

Matrices satisfying these conditions are called \emph{$\Gamma$-circulant}. A proof of Lemma \ref{Cayley spectrum} can be read from \cite[Theorem~1]{Diaconis}, a more general statement in the not-necessarily-abelian setting. (One can also verify Lemma \ref{Cayley spectrum} directly by computing that the putative eigenvalue--eigenvector pairs are in fact eigenvalues and eigenvectors, as well as that the $(\chi(v))_v$ form an orthonormal basis of $\mathbb C^\Gamma$.) 

\section{Minimum Cheeger constant}\label{cheeger constant section}

In this section we determine the minimum value of $h_G$ over all graphs $G$ on $n$ vertices, proving Theorem \ref{cheeger constant theorem}.

\begin{lemma} \label{triangle inequality}
    For any $S\subseteq V(G)$, we have
    $$D(S,S)\le\frac{2(s-1)}{\overline s}D(S,\overline S).$$
Moreover, equality holds if and only if, for every pair $(u,v)$ of distinct vertices in $S$ and every $w\in\overline S$, we have $d(u,v)=d(u,w)+d(w,v)$.
\end{lemma}
\begin{proof}
    For any $w\in\overline S$, we have
    $$D(S,S)=\sum_{u,v\in S} d(u,v)\le\sum_{u\in S,v\in S,u\ne v}(d(u,w)+d(w,v))=2(s-1)\sum_{u\in S}d(u,w).$$
    Therefore
    \[\overline sD(S,S)=\sum_{w\in\overline S}\sum_{u,v\in S}d(u,v)\le 2(s-1)\sum_{u\in S,w\in\overline S}d(u,w)=2(s-1)D(S,\overline S).\qedhere\]
\end{proof}

\begin{lemma} \label{equality case of triangle inequality}
    Let $G$ be a connected graph on at least $7$ vertices with $S\subseteq V(G)$ such that
    \begin{itemize}
        \item $s\in\{\overline s - 1, \overline s,\overline s+1\}$ and
        \item $d(u,v)=d(u,w)+d(w,v)$ for every pair $u,v\in S$ with $u\neq v$ and every $w\in\overline S$. 
    \end{itemize}
    Then $e(S)=0$ and $G[S,\overline S]=K_{S,\overline S}$. 
\end{lemma}
\begin{proof}
    For any distinct $u,v\in S$, choosing an arbitrary $w\in\overline S$ gives $d(u,v)=d(u,w)+d(w,v)\ge 2$. This gives that $e(S)=0$. Since $e(S) = 0$ and $G$ is connected, every vertex in $S$ must have at least one neighbor in $\overline{S}$.
    
    Now we claim that some $u^*, v^*\in S$ have a common neighbor. Since each vertex in $S$ has a neighbor in $\overline{S}$, if $s > \overline s$ then we have a common neighbor by the pigeonhole principle. So suppose $s$ is at most $\overline s$ and that there is no pair with a common neighbor.

    If $s = \overline s$, this implies that every $u\in S$ has exactly one neighbor $u'\in\overline S$. Moreover, each vertex in $\overline S$ is such a neighbor. In particular, there exist $u,v\in S$ such that $u'\sim v'$. Now, since $G[\overline S]$ is connected and $\overline s\ge 3$, then without loss of generality the vertex $v'$ has some neighbor $w'\in\overline S$ distinct from $u'$, connected to some $w\in S$ distinct from $u$ or $v$. The following figure shows the two possible induced subgraphs on these 5 vertices.
     
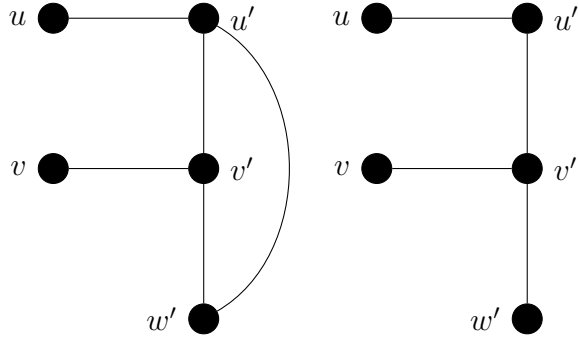
\begin{center}
        \begin{tikzpicture}
        \node[shape=circle,fill=black,draw=black,label=left:$u$] (A) at (0,0) {};
        \node[shape=circle,fill=black,draw=black, label=right:$u'$] (B) at (2,0) {};
        \node[shape=circle,fill=black,draw=black, label=right:$v'$] (C) at (2,-2) {};
        \node[shape=circle,fill=black,draw=black,label=left:$w'$] (D) at (2,-4) {};
        \node[shape=circle,fill=black,draw=black,label=left:$v$] (E) at (0,-2) {};
        \path [-] (A) edge (B);
        \path [-] (E) edge (C);
        \path [-] (B) edge (C);
        \path [-] (C) edge (D);
        \path [-] (B) edge[bend left=60] (D);
        \end{tikzpicture}
        \begin{tikzpicture}
        \node[shape=circle,fill=black,draw=black,label=left:$u$] (A) at (0,0) {};
        \node[shape=circle,fill=black,draw=black, label=right:$u'$] (B) at (2,0) {};
        \node[shape=circle,fill=black,draw=black, label=right:$v'$] (C) at (2,-2) {};
        \node[shape=circle,fill=black,draw=black,label=left:$w'$] (D) at (2,-4) {};
        \node[shape=circle,fill=black,draw=black,label=left:$v$] (E) at (0,-2) {};
        \path [-] (A) edge (B);
        \path [-] (E) edge (C);
        \path [-] (B) edge (C);
        \path [-] (C) edge (D);
        \end{tikzpicture}
    \end{center}
    In this case we have
    \[d(u,w')+d(w',v)\ge 2+2>3=d(u,v),\]
    a contradiction. Finally, assume that $s = \overline s - 1$. Then each vertex in $S$ has at least one neighbor in $\overline{S}$, and if there is a vertex in $S$ with $2$ neighbors in $\overline{S}$ then the same argument as before yields a contradiction. Therefore, we assume that every vertex in $S$ has a unique neighbor in $\overline{S}$ and there is a single vertex $z$ in $\overline{S}$ with no neighbors in $S$. If $e(\overline{S} - \{z\}) > 0$, then the same argument as before yields a contradiction. If $e(\overline{S} - \{z\})= 0$, then since $G$ is connected and $e(S) = 0$, we must have that $G$ is a subdivision of $K_{1,s}$. But this cannot occur because $s\geq 3$: there would be $u,v,w\in S$ adjacent to $u', v', w' \in \overline S$, and in this case $d(u,v) = 4$ while $d(u,w') = d(v, w') = 3$.
    
    We may now assume we can pick some $u^*,v^*\in S$ with $d(u^*,v^*)=2$. For any $w\in S$ we have
    \[2\le d(u^*,w)+d(w,v^*)=d(u^*,v^*)=2,\]
    implying that $u^*\sim w$ and $v^*\sim w$. Thus we have $N(u^*)=N(v^*)=\overline S$. Finally, pick any $x\in S$. We have $d(x,u^*)=2$, and so for any $w\in\overline S$ we have
    \[2=d(x,u^*)\le d(x,w)+d(w,u^*)=d(x,w)+1,\]
    implying $x\sim w$. This proves that $G[S,\overline S]=K_{S,\overline S}$.
\end{proof}

We are now ready to prove Theorem \ref{cheeger constant theorem}. We first show part (i), the case in which the number $n$ of vertices of $G$ is even.

\begin{theorem} \label{minimum Cheeger constant even}
    Let $G$ be a graph on $n=2m$ vertices. Then $h_G\ge\frac{n}{3n-4}$ with equality if and only if $K\simeq K_{m,m}$.
\end{theorem}
\begin{proof}
If $n \in \{2,4\}$ it is easy to check by hand that the theorem holds, so suppose $n\geq 6$.  Suppose $G$ has Cheeger partition $S,\overline S$, and assume without loss of generality that $\vol S\le\vol\overline S$.\\
    \textit{Case 1:} $s\le\overline s$. Then $s\le m$ and $\overline s\ge m$. By Lemma \ref{triangle inequality}, we have
    $$D(S,S)\le\frac{2(s-1)}{\overline s}D(S,\overline S)\le\frac{2(m-1)}{m}D(S,\overline S)=\frac{2n-4}{n}D(S,\overline S).$$
    Therefore, using (\ref{volume equality}),
    \[h_G=h(S)=\frac{D(S,\overline S)}{\vol S}=\frac{D(S,\overline S)}{D(S,S)+D(S,\overline S)}\ge\frac{D(S,\overline S)}{\frac{2n-4}{n}D(S,\overline S)+D(S,\overline S)}=\frac{n}{3n-4}.\]
    If equality holds, then in the application of Lemma \ref{triangle inequality} we have $d(u,v)=d(u,w)+d(w,v)$ for every distinct $u,v\in S$ and $w\in\overline S$. Moreover, we must have $s=\overline s=m$. Thus, by Lemma \ref{equality case of triangle inequality}, we have that $e(S)=0$ and $G[S,\overline S]=K_{S,\overline S}$. Then $\vol S\le\vol\overline S $ implies
    $$2\binom{s}{2}=D(S,S)\le D(\overline S,\overline S)\le 2\binom{\overline s}{2}=2\binom{s}{2}$$
    so equality holds throughout and $e(\overline S)=0$, completing the proof.
    \\
    \textit{Case 2:} $s>\overline s$. Then we have
        $$D(\overline S,\overline S)\le\frac{2(\overline s-1)}{s}D(S,\overline S)<\frac{2n-4}{n}D(S,\overline S)$$
    and so
    \begin{align*}
    h_G=h(S)=\frac{D(S,\overline S)}{\vol S}\ge\frac{D(S,\overline S)}{\vol\overline S}&=\frac{D(S,\overline S)}{D(\overline S,\overline S)+D(\overline S,S)}\\
    &>\frac{D(S,\overline S)}{\frac{2n-4}{n}D(S,\overline S)+D(S,\overline S)}=\frac{n}{3n-4}.
    \end{align*}
    As we have strict inequality in this case, this concludes the proof.
\end{proof}

We now verify Theorem \ref{cheeger constant theorem}(ii), the odd case.

\begin{theorem} \label{minimum Cheeger constant odd}
    Let $G$ be a graph on $n=2m+1$ vertices. Then $h_G\ge\frac{n+1}{3n-5}$. If $n\ne 5$, then equality holds if and only if $G$ is obtained from $K_{m,m+1}$ by embedding at most $2m+1$ edges in the larger part.
\end{theorem}
\begin{proof}
If $n\in \{3,5\}$ one may check the conclusion by hand. Therefore we may assume that $n\geq 7$. 

    Let $G$ have Cheeger partition $S,\overline S$ and assume without loss of generality that $\vol S\le\vol\overline S$. \\
    \textit{Case 1:} $s\le\overline s$. Then $s\le m$ and $\overline s\ge m+1$. Then by Lemma \ref{triangle inequality}, we have
    $$D(S,S)\le\frac{2(s-1)}{\overline s}D(S,\overline S)\le\frac{2m-2}{m+1}D(S,\overline S)=\frac{2n-6}{n+1}D(S,\overline S).$$
    Thus
    \[h_G=h(S)=\frac{D(S,\overline S)}{\vol S}=\frac{D(S,\overline S)}{D(S,S)+D(S,\overline S)}\ge\frac{D(S,\overline S)}{\frac{2n-6}{n+1}D(S,\overline S)+D(S,\overline S)}=\frac{n+1}{3n-5}.\]
    If equality holds, then $s=m$, $\overline s=m+1$, and for any distinct $u,v\in S$ and $w\in\overline S$ we have $d(u,v)=d(u,w)+d(w,v)$. By Lemma \ref{equality case of triangle inequality}, this implies that $e(S)=0$ and $G[S,\overline S]=K_{S,\overline S}$. Now we have
    $$2m(m-1)+D(S,\overline S)=D(S,S)+D(S,\overline S)=\vol S\le\vol\overline S=D(\overline S,\overline S)+D(S,\overline S)$$
    and hence
    $$2m(m-1)\le D(\overline S,\overline S)=2(m+1)m-2e(\overline S)$$
    and $e(\overline S)\le 2m$. If in fact $e(\overline S)\leq 2m$, then the partition $S\sqcup\overline S$ of the vertex set of $G$ gives $h(S)=\frac{n+1}{3n-5}$, and so the equality $h_G=\frac{n+1}{3n-5}$ holds in this case.\\
    \textit{Case 2:} $s\ge\overline s$. Then by Lemma \ref{triangle inequality}, we have
    $$D(\overline S,\overline S)\le\frac{2(\overline s-1)}{ s}D(S,\overline S)\le\frac{2n-6}{n+1}D(S,\overline S).$$
    Thus
    \begin{align*}
    h_G=h(S)=\frac{D(S,\overline S)}{\vol S}\ge\frac{D(S,\overline S)}{\vol\overline S}
    &=\frac{D(S,\overline S)}{D(\overline S,\overline S)+D(S,\overline S)}\\
    &\ge\frac{D(S,\overline S)}{\frac{2n-6}{n+1}D(S,\overline S)+D(S,\overline S)}=\frac{n+1}{3n-5}.
    \end{align*}
    If equality holds then we have $s=m+1,\overline s=m$, and $d(u,v)=d(u,w)+d(w,v)$ for every $u,v\in\overline S,w\in S$. Then by Lemma \ref{equality case of triangle inequality}, we have $e(\overline S)=0$ and $G[S,\overline S]=K_{S,\overline S}$. Moreover, $\vol S=\vol\overline S$ and so
    $$2(m+1)m-2e(S)+D(S,\overline S)=\vol S=\vol\overline S=2m(m-1)+D(S,\overline S)$$
    and we obtain $e(S)=2m$. If in fact $e(S)=2m$, then the partition $S\sqcup\overline S$ of the vertex set of $G$ gives $h(S)=\frac{n+1}{3n-5}$, and so the equality $h_G=\frac{n+1}{3n-5}$ holds in this case.
\end{proof}

Combining Theorems \ref{minimum Cheeger constant even} and \ref{minimum Cheeger constant odd} proves Theorem \ref{cheeger constant theorem}. 

\begin{corollary} \label{1/3 lower bound}
    For any graph $G$, we have $h_G\ge\frac{1}{3}$.
\end{corollary}

\section{Spectral gap}\label{eigenvalue section}

In this section we prove Theorems \ref{spectral gap theorem} and \ref{2/3 theorem}. Our proofs of both of these results will be quite algebraic, and are inspired by considering the computation of $\partial_2$ as an optimization problem. In the case of Theorem \ref{spectral gap theorem}, as we have no structural information about the eigenvectors of $\DL$, the problem is essentially a semidefinite optimization problem. We will first reduce this to a question about verifying the nonnegativity of a family of linear forms in the distances $d(u,v)$. We will then show by explicit computation that it is possible to write each of these linear forms as a nonnegative linear combination of $d(u,v)+d(u,w)-d(v,w)\geq 0$ for various $u,v,w$. 

In the case of Theorem \ref{2/3 theorem}, we will be able to use the symmetry of Cayley graphs to reduce the problem of bounding $\partial_2$ to verifying the nonnegativity of some simple (and more explicit, compared to the proof of Theorem \ref{spectral gap theorem}) linear forms in $d(u,v)$. In some cases, we are able to perform this verification combinatorially, while in other cases we use an algebraic procedure based on nonnegativity of some trigonometric polynomials, similar to what we use to prove Theorem \ref{spectral gap theorem}. The explicit structure of these inequalities, which comes from character theory for abelian groups, is the reason we are able to attain better bounds in the Cayley graph setting than the general setting.

Throughout this section, all sums will be over elements of $V(G)$ unless otherwise stated.

\subsection{Arbitrary graphs: proof of Theorem \ref{spectral gap theorem}} We first prove Theorem \ref{spectral gap theorem}. We begin by stating a technical lemma from which we may easily deduce the theorem, and which presents the ``optimization'' characterization of $\partial_2$ in a clean form.

\begin{lemma}\label{semidefinite inequalities} 
Let $G$ be a graph with distance function $d\colon V(G)\times V(G)\to\mathbb Z_{\geq 0}$. Let $y\in\RR^{V(G)}$ be a vector satisfying $\sum_vy_v=0$. Then we have the inequality
\[\sum_{u,v\in V(G)}(y_u^2-(1+2\sqrt 2)y_uy_v+y_v^2)d(u,v)\geq 0.\]
\end{lemma}

\begin{proof}[Proof of Theorem \ref{spectral gap theorem} assuming Lemma \ref{semidefinite inequalities}] Rearranging the conclusion of Lemma \ref{semidefinite inequalities} gives
\begin{align*}
\left(\sqrt2+\frac12\right)\sum_{u,v}d(u,v)(y_u-y_v)^2&=\left(\sqrt2+\frac12\right)\sum_{u,v}d(u,v)(y_u^2-2y_uy_v+y_v^2)\\
&\geq\left(\sqrt2-\frac12\right)\sum_{u,v}d(u,v)(y_u^2+y_v^2)\\
&=\left(\sqrt2-\frac12\right)2\sum_ut(u)y_u^2
\end{align*}
for every vector $y$ with $y\perp\mathbf 1$. As a result, we have
\[\min\limits_{y\perp\mathbf{1}}\frac{\sum_{u,v}d(u,v)(y_u-y_v)^2}{2\sum_{u}t(u)y_u^2}\geq\frac{\sqrt2-\frac12}{\sqrt2+\frac12}=\frac{9-4\sqrt2}7.\]
The result now follows from Lemma \ref{Courant-Fischer}.
\end{proof}

What remains is to prove Lemma \ref{semidefinite inequalities}. As mentioned in the introduction of this section, we essentially only have the triangle inequality at our disposal, and so we will obtain Lemma \ref{semidefinite inequalities} as a (nonnegative) linear combination of the inequalities $d(u,v)+d(u,w)-d(v,w)\geq 0$ for various $u,v,w$. The following lemma makes this strategy explicit, and simplifies it via the use of some symmetries.

\begin{lemma}\label{lin comb} Let $G$ be a graph with distance function $d\colon V(G)\times V(G)\to\mathbb Z_{\geq 0}$, and let $z\in\RR^{V(G)\times V(G)}$. Suppose there exists a weight function $\mu\colon V(G)^3\to\RR$ satisfying the following conditions:
\begin{enumerate}[(i)]
    \item For each $u,v,w\in V(G)$, we have $\mu(u,v,w)=\mu(v,u,w)$.
    
    \item For each $u,v,w\in V(G)$, we have $\mu(u,v,w)+\mu(u,w,v)\geq 0$.

    \item For each $u,v\in V(G)$, we have
    \[\sum_w\mu(u,v,w)=z_{u,v}.\]
\end{enumerate}
Then
\[\sum_{u,v}z_{u,v}d(u,v)\geq 0.\]
\end{lemma}
\begin{proof} We use the triangle inequality, expand, change variables, and group terms:
\begin{align*}
0&\leq\sum_{u,v,w}\frac{\mu(u,v,w)+\mu(u,w,v)}2\big(d(u,v)+d(u,w)-d(v,w)\big)\\
&=\sum_{u,v,w}\mu(u,v,w)\big(d(u,v)+d(u,w)-d(v,w)\big)\\
&=\sum_{u,v,w}\mu(u,v,w)d(u,v)+\sum_{u,v,w}\mu(u,v,w)d(u,w)-\sum_{u,v,w}\mu(u,v,w)d(v,w)\\
&=\sum_{u,v,w}\mu(u,v,w)d(u,v)+\sum_{u,v,w}\mu(u,w,v)d(u,v)-\sum_{u,v,w}\mu(w,u,v)d(u,v)\\
&=\sum_{u,v,w}\mu(u,v,w)d(u,v)\\
&=\sum_{u,v}d(u,v)\sum_w\mu(u,v,w)=\sum_{u,v}z_{u,v}d(u,v).
\end{align*}
In the first line, we used the triangle inequality and condition (ii). In the second, we used the symmetry of $d(u,v)+d(u,w)-d(v,w)$ in $\{v,w\}$. In the fourth, we swapped ``$v$'' and ``$w$'' in the second sum and cycled ``$u$''$\to$``$w$''$\to$``$v$''$\to$``$u$'' in the third sum. In the fifth line, we used condition (i), and in the sixth line, we used condition (iii). 
\end{proof}

What remains now is simply to exhibit weights $\mu(\cdot,\cdot,\cdot)$ which satisfy the conditions of Lemma \ref{lin comb} with $z_{u,v}=y_u^2-(1+2\sqrt2)y_uy_v+y_v^2$.

\begin{proof}[Proof of Lemma \ref{semidefinite inequalities}] We may assume $\|y\|^2=1$. Define a function $f\colon\RR^3\to\RR$ by
\[f(\alpha,\beta,\gamma)=(\alpha^2-(1+2\sqrt2)\alpha\beta+\beta^2)\gamma^2+\alpha\beta\gamma(\alpha+\beta)\]
and set $\mu(u,v,w)=f(y_u,y_v,y_w)$. Since $f$ is symmetric in $\{\alpha,\beta\}$, property (i) of Lemma \ref{lin comb} holds for $\mu$. To verify property (ii), we may expand
\begin{align*}
f(\alpha,\beta,\gamma)+f(\alpha,\gamma,\beta)
&=\alpha^2(\beta^2+\gamma^2)+2\alpha^2\beta\gamma+2\beta^2\gamma^2-2\sqrt2\alpha\beta\gamma(\beta+\gamma)\\
&=(\alpha\beta+\alpha\gamma-\sqrt2\beta\gamma)^2\geq 0.
\end{align*}
Finally, for (iii), we compute for each $u,v\in V(G)$ that
\begin{align*}
\sum_{w}\mu(u,v,w)
&=\sum_{w}\left((y_u^2-(1+2\sqrt2)y_uy_v+y_v^2)y_w^2+y_uy_vy_w(y_u+y_v)\right)\\
&=(y_u^2-(1+2\sqrt2)y_uy_v+y_v^2)\sum_{w}y_w^2+y_uy_v(y_u+y_v)\sum_{w}y_w\\
&=y_u^2-(1+2\sqrt2)y_uy_v+y_v^2,
\end{align*}
where we have used that $\|y\|^2=1$ and $\sum_{w}y_w=0$. The result now follows from applying Lemma~\ref{lin comb}.  
\end{proof}

\subsection{Cayley graphs on abelian groups}

We now prove Theorem \ref{2/3 theorem}, lower-bounding $\partial_2$ for Cayley graphs on abelian groups. The symmetry of these graphs affords to $\DL$ an explicit description of the eigenvalues and eigenvectors. This simplifies our analysis considerably, and allows us to obtain strong lower bounds on $\partial_2$.

Henceforth, we fix an abelian group $\Gamma$ and a Cayley graph $G$ on $\Gamma$. Write $\widehat\Gamma$ for the set of characters of $\Gamma$. Letting $d$ be the metric $G$ induces on $\Gamma$, we have $d(u,v)=d(0,v-u)$ for each $u,v\in\Gamma$. In particular, the transmission vector of $G$ is constant:
\[t(v)=\sum_{u\in\Gamma}d(u,v)=\sum_{u\in\Gamma}d(0,v-u)=t(0).\]
Write $t_0:=t(0)$. Define $f\colon\Gamma\to\RR$ by $f(0)=1$ and $f(v)=-\frac{d(0,v)}{t_0}$. In the notation of Lemma \ref{Cayley spectrum}, we have $\DL=M_f$, and so
\begin{equation}\label{Cayley eigs}
\{\partial_1,\ldots,\partial_{|\Gamma|}\}=\left\{1-\frac1{t_0}\sum_{v\in\Gamma}d(0,v)\chi(v):\chi\in\widehat\Gamma\right\}.
\end{equation}
Given this characterization, we may reduce Theorem \ref{2/3 theorem} to the following statement, which resembles Lemma \ref{semidefinite inequalities}.

\begin{lemma}\label{Cayley optimization} Let $\Gamma$ be an abelian group, and let $(d_u)_{u\in\Gamma}$ be any vector of nonnegative reals which satisfies $d_v=d_{-v}$ and $d_{u+v}\leq d_u+d_v$ for $u,v\in\Gamma$. Let $\chi\in\widehat{\Gamma}$ be a nontrivial character of $\Gamma$.
\begin{enumerate}
    \item[(a)] If $\chi^2=1$, i.e.~if the image of $\chi$ lies in $\{-1,1\}$, then
    \[\sum_{v\in\Gamma}d_v(1-3\chi(v))\geq 0.\]

    \item[(b)] If $\chi^2\neq1$, then
    \[\sum_{v\in\Gamma}d_v(1-C_1\chi(v))\geq 0,\]
    where $C_1\approx 3.554$ is the largest root of the polynomial $4x^4 - 4x^3 - 31x^2 - 20x + 4$.
\end{enumerate}
(The symmetry of $d_v$ ensures that both given sums are in fact real.)
\end{lemma}

\begin{proof}[Proof of Theorem \ref{2/3 theorem} assuming Lemma \ref{Cayley optimization}] Apply Lemma \ref{Cayley optimization} to $d_v=d(0,v)$, which can be easily seen to satisfy the necessary properties. Conditions (a) and (b) then rearrange to
\[\sum_{v\in\Gamma}d(0,v)\chi(v)\leq c_\chi\sum_{v\in\Gamma}d(0,v)=c_\chi t_0,\]
where $c_\chi=1/3$ if $\chi^2=1$ and $c_\chi=1/C_1$ if $\chi^2\neq 1$. This implies
\begin{equation}\label{chi intermediate}
1-\frac1{t_0}\sum_{v\in\Gamma}d(0,v)\chi(v)\geq 1-c_\chi.
\end{equation}
Now, we use (\ref{Cayley eigs}). The eigenvalue $\partial_1=0$ comes from the trivial character $\chi(\cdot)=1$, and so (\ref{chi intermediate}) implies
\[\partial_2\geq\min_{\substack{\chi\in\widehat\Gamma\\\chi\text{ nontriv.}}}(1-c_\chi).\]
We have $1-c_\chi\geq2/3$ for any nontrivial $\chi$; this is enough to prove the first part of Theorem \ref{2/3 theorem}. If, additionally, $\Gamma$ has odd order, then $\chi^2\neq 1$ for any nontrivial $\chi$, and so $c_\chi=1/C_1$. We conclude $\partial_2\geq 1-1/C_1>0.718$, as desired.
\end{proof}

We have now reduced Theorem \ref{2/3 theorem} to a ``purely analytic'' statement like that of Lemma \ref{semidefinite inequalities}. This statement has two main advantages over Lemma \ref{semidefinite inequalities}. First, as it presents a single sum instead of a double sum, we have fewer parameters to juggle in our applications of the triangle inequality. Second, and more importantly, the vector $(\chi(v))_v$ has rich structure which we may use.

We next give a short combinatorial argument establishing Lemma \ref{Cayley optimization}(a).

\begin{proof}[Proof of Lemma \ref{Cayley optimization}(a)] Since the image of $\chi$ lies in $\{-1,1\}$, the character $\chi$ induces a partition of $\Gamma$ into two subsets $S:=\chi^{-1}(1)$ and $\overline S:=\chi^{-1}(-1)$. Select $u_0\in \overline S$ for which $d_{u_0}$ is minimal. For any $u\in\overline S$, $\chi(u_0+u)=\chi(u_0)\chi(u)=(-1)^2=1$, and so $S=\{u_0+u:u\in \overline S\}$, and we have $|S|=|\overline S|$. We can now rewrite the desired inequality as
\[\sum_{v\in S}d_v\leq 2\sum_{u\in\overline S}d_u.\]
 We conclude
\[\sum_{v\in S}d_v=\sum_{u\in R}d_{u_0+u}\leq\sum_{u\in\overline S}(d_{u_0}+d_u)=|\overline S|d_{u_0}+\sum_{u\in\overline S}d_u\leq 2\sum_{u\in\overline S}d_u,\]
where the last inequality is by the minimality of $d_{u_0}$, as desired.
\end{proof}

What remains is to show Lemma \ref{Cayley optimization}(b). We will use a procedure similar to that used to prove Lemma \ref{semidefinite inequalities}, wherein we sum the inequalities $d_u+d_v-d_{u+v}\geq 0$ with various weights. We will not make explicit a general version of the choice of weights (cf.~Lemma~\ref{lin comb}), but one can be written down. 

Fix a character $\chi$ of $\Gamma$ with $\chi^2\neq 1$, and let $\theta\in(\RR/2\pi\ZZ)^\Gamma$ be so that $\chi(v)=e^{i\theta_v}$ for each $v\in\Gamma$. Let $A$ and $B$ be real parameters which we will tune later. Define
\[g(\alpha,\beta)=\left(A+B\cos(\alpha+\beta)-B\frac{\cos\alpha+\cos\beta}{\sqrt2}\right)^2.\]
We may now rearrange
\begin{align}
0&\leq\frac1{|\Gamma|}\sum_{u,v\in\Gamma}(d_u+d_v-d_{u+v})g(\theta_u,\theta_v)\notag\\
&=\sum_{u\in\Gamma}d_u\frac1{|\Gamma|}\left(\sum_{v\in\Gamma}g(\theta_u,\theta_v)\right)+\sum_{v\in\Gamma}d_v\frac1{|\Gamma|}\left(\sum_{u\in\Gamma}g(\theta_u,\theta_v)\right)-\sum_{w\in\Gamma}d_w\frac1{|\Gamma|}\left(\sum_{v\in\Gamma}g(\theta_{w-v},\theta_v)\right)\notag\\
&=\sum_{u\in\Gamma}d_u\left(\frac1{|\Gamma|}\sum_{v\in\Gamma}\big(2g(\theta_u,\theta_v)-g(\theta_u-\theta_v,\theta_v)\big)\right).\label{full g sum}
\end{align}
We relegate the remainder of the computation to the following lemma.

\begin{lemma}\label{g computation} For each $\varphi\in\RR/2\pi\ZZ$, we have
\[\frac1{|\Gamma|}\sum_{v\in\Gamma}\big(2g(\varphi,\theta_v)-g(\varphi-\theta_v,\theta_v)\big)=(A^2+B^2)-\left(2AB(1+\sqrt2)+B^2\left(\sqrt2+\frac12\right)\right)\cos\varphi.\]
\end{lemma}
\begin{proof} This is simply some computation using orthogonality of characters. We begin by noting that, for any fixed angle $\gamma$ and any integer $m$,
\begin{equation}\label{cosine sums}
\frac1{|\Gamma|}\sum_{v\in\Gamma}\cos(\gamma+m\theta_v)=\operatorname{Re}\left(e^{i\gamma}\cdot\frac1{|\Gamma|}\sum_{v\in\Gamma}\chi(v)^m\right)=\begin{cases}\cos(\gamma)&\text{if }\chi^m=1\\0&\text{otherwise,}\end{cases}
\end{equation}
where we have used orthogonality of characters in the last equality. One can expand
\begin{align*}
g(\alpha,\beta)
&=(A^2+B^2)-\left(AB\sqrt2+\frac{B^2}{\sqrt2}\right)(\cos\alpha+\cos\beta)\\
&\ \ \ \ \ \ \ \ \ \ +\left(2AB+\frac{B^2}2\right)\cos(\alpha+\beta)+\frac{B^2}4(\cos2\alpha+\cos2\beta)\\
&\ \ \ \ \ \ \ \ \ \ +\frac{B^2}2(\cos(\alpha-\beta)+\cos(2\alpha+2\beta))-\frac{B^2}{\sqrt2}(\cos(2\alpha+\beta)+\cos(\alpha+2\beta)).
\end{align*}
When we sum the above form for $g(\varphi,\theta_v)$ over all $v\in\Gamma$, using (\ref{cosine sums}), only the $\{1,\cos\alpha,\cos2\alpha\}$ terms remain, and we get
\begin{equation}\label{first g sum}\frac1{|\Gamma|}\sum_{v\in\Gamma}g(\varphi,\theta_v)=(A^2+B^2)-\left(AB\sqrt2+\frac{B^2}{\sqrt2}\right)\cos\varphi+\frac{B^2}4\cos2\varphi.
\end{equation}
When we sum the above form for $g(\varphi-\theta_v,\theta_v)$, only the $\cos(\alpha+\beta)$ and $\cos(2(\alpha+\beta))$ terms remain, and we get
\begin{equation}\label{second g sum}
\frac1{|\Gamma|}\sum_{v\in\Gamma}g(\varphi-\theta_v,\theta_v)=(A^2+B^2)+\left(2AB+\frac{B^2}2\right)\cos\varphi+\frac{B^2}2\cos2\varphi.
\end{equation}
Subtracting (\ref{second g sum}) from twice (\ref{first g sum}), the $\cos2\varphi$ terms cancel, and we obtain the result.
\end{proof}

\begin{proof}[Proof of Lemma \ref{Cayley optimization}(b)] By (\ref{full g sum}) and Lemma \ref{g computation}, we have
\[\sum_{u\in\Gamma}d_u\left((A^2+B^2)-\left(2AB(1+\sqrt2)+B^2\left(\sqrt2+\frac12\right)\right)\cos\theta_u\right)\geq 0\]
for any real $A,B$. The maximum value of the quadratic form
\[2AB(1+\sqrt2)+B^2\left(\sqrt2+\frac12\right)\]
on the circle $A^2+B^2=1$ is exactly the largest root of the polynomial $4x^4-4x^3-31x^2-20x+4$, which is $C_1\approx 3.554$. (This is attained at $A\approx 0.562$ and $B\approx 0.827$.) Taking these values for $A$ and $B$, we conclude the result.
\end{proof}

\section{Concluding Remarks}

The main open question remaining is to prove Conjecture \ref{2/3 conjecture}. It would also be interesting to give conditions under which the $2/3$ can be improved, as in Theorem \ref{2/3 theorem}. In particular, it would be interesting to determine families of graphs in which one has $\partial_2 \to 1$. One natural family to investigate, in view of Theorem \ref{2/3 theorem}, is the family of Cayley graphs on abelian groups with order relatively prime to $\ell!$, as $\ell$ grows. Other future work might include considering $\partial_k$ for fixed $k>2$ as in \cite{LRT}, studying the spectral gap of the non-normalized distance Laplacian or of the distance matrix itself as in \cite{HC}, or using the eigenvector for $\partial_2$ to say something stronger as in \cite{franklin}.

We conclude the paper by noting that the proofs of Theorems \ref{spectral gap theorem} and \ref{2/3 theorem} apply, unchanged, to the more general setting of finite metric spaces. Let $X$ be a finite set and $d\colon X\times X\to\RR_{\geq 0}$ any metric on $X$. (The setting of graphs is recovered by setting $X=V(G)$ and $d$ to be the shortest-path metric.) As with graphs, we may define the distance matrix $\mathcal D(X,d)$ of $(X,d)$ by $\mathcal D_{uv}=d(u,v)$, the transmission vector $t$ by $t(u)=\sum_{v\in X}d(u,v)$, the transmission matrix $T(X,d)=\operatorname{diag}(t)$, and the normalized distance Laplacian by
\[\DL(X,d)=T(X,d)^{-1/2}(T(X,d)-\cD(X,d))T(X,d)^{-1/2}.\]
The spectrum of $\DL(X,d)$, like that of $\DL(G)$, is contained in $[0,2]$. The analogues of Theorems \ref{spectral gap theorem} and \ref{2/3 theorem} are as follows. 

\begin{theorem}\label{spectral gap theorem metric spaces} For any finite metric space $(X,d)$, the spectral gap of $\DL(X,d)$ is at least $\frac{9-4\sqrt2}7\approx 0.478$.     
\end{theorem}

\begin{theorem}\label{2/3 theorem metric spaces} Let $\Gamma$ be a finite abelian group, and let $d$ be a metric on $\Gamma$ satisfying $d(u,v)=d(u+w,v+w)$ for every $u,v,w\in\Gamma$. Then $\DL(\Gamma,d)$ has spectral gap at least $2/3$. Moreover, if $\Gamma$ has odd order, then the spectral gap of $\DL(\Gamma,d)$ exceeds $0.718$.    
\end{theorem}

\bibliographystyle{plain}
\bibliography{bib}
\end{document}